\documentclass[twoside,11pt]{article}

\usepackage{blindtext}

\usepackage{jmlr2e}

\usepackage{amsfonts}
\usepackage{amsmath}
\usepackage{amssymb}
\usepackage{bm}
\usepackage{caption}
\usepackage{subfigure}
\usepackage{graphicx}
\usepackage{algorithm}
\usepackage{algpseudocode}
\usepackage{multirow}
\usepackage{bbm}
\usepackage{xcolor}
\usepackage{booktabs}
\usepackage{yfonts}
\usepackage{enumitem}
\usepackage{adjustbox}
\usepackage{mathtools}
\usepackage[left,mathlines]{lineno}

\newcommand{\Bx}{\bm{x}}
\newcommand{\Be}{\bm{e}}
\newcommand{\Bhe}{\bm{\hat{e}}}

\newcommand{\Bp}{\bm{p}}

\newcommand{\R}{\ifmmode\mathbb{R}\else$\mathbb{R}$\fi}
\newcommand{\C}{\ifmmode\mathbb{C}\else$\mathbb{C}$\fi}
\newcommand{\N}{\ifmmode\mathbb{N}\else$\mathbb{N}$\fi}
\newcommand{\Q}{\ifmmode\mathbb{Q}\else$\mathbb{Q}$\fi}
\newcommand{\Z}{\ifmmode\mathbb{Z}\else$\mathbb{Z}$\fi}

\DeclareMathOperator*{\argmin}{arg\,min}

\usepackage{hyperref}

\usepackage{lastpage}

\firstpageno{1}

\begin{document}

\title{Finite Expression Method for Solving High-Dimensional Partial Differential Equations}

\author{\name Senwei Liang \email senwei.liang@ttu.edu \\
       \addr Department of Mathematics and Statistics\\
       Texas Tech University\\
       Lubbock, TX 79409, USA
       \AND
       \name Haizhao Yang \email hzyang@umd.edu \\
       \addr Department of Mathematics and Department of Computer Science\\
       University of Maryland, College Park\\
       College Park, MD 20742, USA}

\editor{}

\maketitle

\begin{abstract}
Designing efficient and accurate numerical solvers for high-dimensional partial differential equations (PDEs) remains a challenging and important topic in computational science and engineering, mainly due to the ``curse of dimensionality" in designing numerical schemes that scale in dimension. This paper introduces a new methodology that seeks an approximate PDE solution in the space of functions with finitely many analytic expressions and, hence, this methodology is named the finite expression method (FEX). It is proved in approximation theory that FEX can avoid the curse of dimensionality. As a proof of concept, a deep reinforcement learning method is proposed to implement FEX for various high-dimensional PDEs in different dimensions, achieving high and even machine accuracy with a memory complexity polynomial in dimension and an amenable time complexity. An approximate solution with finite analytic expressions also provides interpretable insights into the ground truth PDE solution, which can further help to advance the understanding of physical systems and design postprocessing techniques for a refined solution.
\end{abstract}


\section{Introduction}

Partial differential equations (PDEs) play a fundamental role in scientific fields for modeling diverse physical phenomena, including diffusion~\citep{philibert2005one,kirkwood1960flow}, fluid dynamics~\citep{acheson1991elementary,shinbrot2012lectures} and quantum mechanics~\citep{feynman1965feynman,landau2013quantum}. Developing efficient and accurate solvers for numerical solutions to high-dimensional PDEs remains an important and challenging topic~\citep{weinan2021algorithms}. Many traditional solvers, such as finite element method (FEM)~\citep{reddy2004introduction} and finite difference~\citep{grossmann2007numerical}, are usually limited to low-dimensional domains since the computational cost increases exponentially in the dimension~\citep{weinan2021algorithms,fuli2022}. Recently, neural networks (NNs) as mesh-free parameterization are widely employed in solving high-dimensional PDEs~\citep{Weinan2017,Han2018,Khoo2017SolvingPP,RAISSI2019686,sirignano2018dgm,yu2018deep} and control problems \citep{Han2016DeepLA}. In theory, NNs have the capability of approximating various functions well and lessening the curse of dimensionality \citep{Shen4,Shen5,Yarotsky2021ElementarySA,Shen2021DeepNA,Jiao2021DeepNN}. Yet the highly non-convex objective function and the spectral bias toward fitting a smooth function in NN optimization make it difficult to achieve high accuracy~\citep{FP,cao2019towards,ronen2019convergence}. In practice, NN-based solvers can hardly achieve a highly accurate solution even when the true solution is a simple function, especially for high-dimensional problems~\citep{yu2018deep,Liu2020Jul}. 
In addition, NN parameterization may still require large memory and high computational cost for high-dimensional problems~\citep{bianco2018benchmark}. Finally, numerical solutions from both traditional solvers and NN-based solvers lack interpretability, e.g., the dependence of the solution on variables is not readily apparent from the numerical solutions.

\begin{figure}
		\centering
		\includegraphics[width=0.99\linewidth]{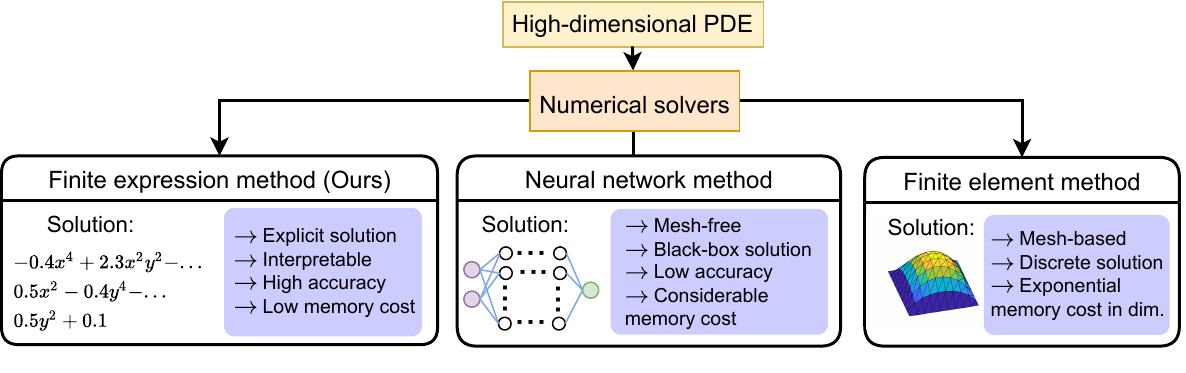}
		\caption{Overview of various numerical solvers for PDEs. In our finite expression method, we aim to find a PDE solution as a mathematical expression with finitely many operators. The resulting solution can reproduce the true solution and achieve high, even machine-level, accuracy. Furthermore, the mathematical expression has  memory complexity polynomial in dimension and provides interpretable insights to advance understanding of physical systems and aid in designing postprocessing techniques for solution refinement. }
		\label{fig:concept}
\end{figure}
	
In this paper, we propose the finite expression method (FEX), a methodology that aims to find a solution in the function space of mathematical expressions with finitely many operators. Compared with the NN and FEM methods, our FEX enjoys the following advantages (summarized in Figure \ref{fig:concept}): (1) The expression can reproduce the true solution and achieve high, even machine-level, accuracy. (2) The expression requires low memory (a line of string) for solution storage and low computational cost for evaluation on new points. (3) The expression has good interpretability in an explicit and readable form. Moreover, from an approximation theory perspective detailed in Section \ref{sec:theory}, the expression in FEX is capable of avoiding the curse of dimensionality in theory. 

In FEX, we formulate the search for mathematical expressions as a combinatorial optimization (CO) involving both discrete and continuous variables. While many techniques can be used to solve the CO in FEX, we provide a numerically effective implementation based on reinforcement learning (RL). Traditional algorithms (e.g., genetic programming and simulated annealing \citep{murray2012modelling}) address CO problems by employing hand-crafted heuristics that are highly dependent on specific problem formulations and domain knowledge \citep{bello2016neural,cheung2019thompson,mazyavkina2021reinforcement}. However, RL has emerged as a popular and versatile tool for learning how to construct CO solutions based on reward feedback, without the need for extensive heuristic design. The success of RL applications in CO, such as automatic algorithm design \citep{ramachandran2018searching,bello2017neural,co-reyes2021evolving} and symbolic regression \citep{petersen2021deep,pmlr-v139-landajuela21a}, has inspired us to seek mathematical expression solutions with RL.  Specifically, in our implementation, the mathematical expression is represented by a binary tree, where each node is an operator along with parameters (scaling and bias) as shown in Figure \ref{fig:trees} and further explained in Section \ref{sec:elementary}. The objective function we aim to minimize is a functional, and its minimization leads to the solution of the PDE, as described in Section \ref{sec:error}. Consequently, our problem involves the minimization of both discrete operators and continuous parameters embedded within the tree structure.  Optimizing both discrete and continuous variables simultaneously is inherently difficult. We propose a search loop for this CO as depicted in Figure \ref{fig:workflow}. Our idea is to first identify good operator selections that have a high possibility of recovering the true solution structure, and then optimize the parameters. The proposals of operator selections are drawn from a controller which will be updated via the policy gradient \citep{petersen2021deep} iteratively. In Section \ref{sec:numericalexp}, we numerically demonstrate the ability of this RL-based implementation to find mathematical expressions that solve high-dimensional PDEs with high, even machine-level accuracy. Furthermore, FEX provides interpretable insights into the ground-truth PDE solutions, which can further advance understanding of physical systems and aid in designing postprocessing techniques to refine solutions.

\begin{figure}[ht]
		\centering
		\includegraphics[width=0.80\linewidth]{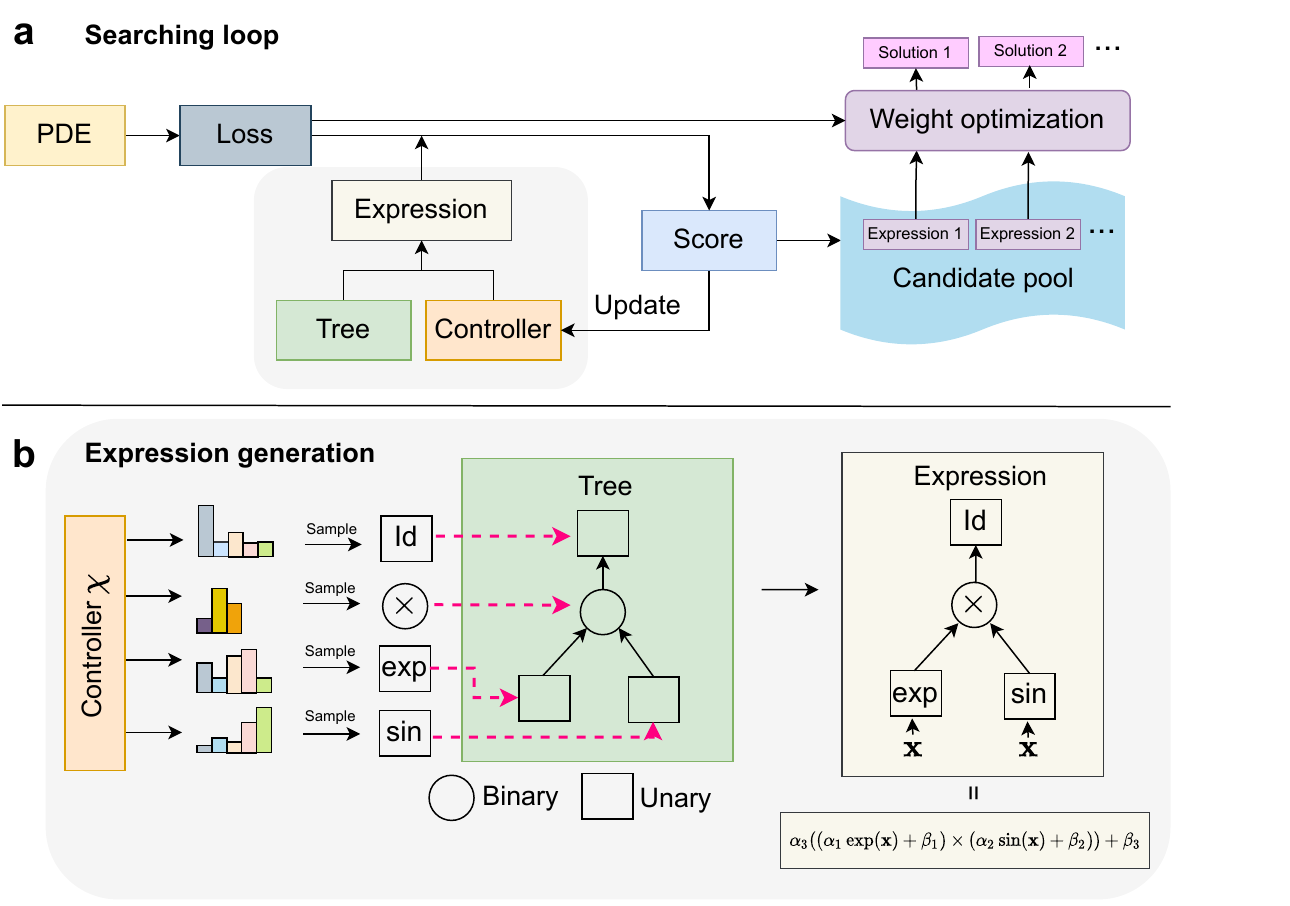}
		\caption{Representation of the components of our FEX implementation. (a) The searching loop for the finite expression solution consists of expression generation, score computation, controller update, and candidate optimization. (b) Depiction of the expression generation with a binary tree and a controller $\bm{\chi}$. The controller outputs probability mass functions for each node in the tree, and these functions are used to sample node values. The final expression, which incorporates learnable scaling and bias parameters, is constructed based on the predefined tree structure and the sampled node values. }
		\label{fig:workflow}
\end{figure}

\section{An Abstract Methodology of Finite Expression Method}
\label{sec:abstractframework}
The goal of FEX is to find a mathematical expression to solve a PDE. This section will formally define the function space of mathematical expressions and formulate the problem of FEX as a CO. 

\subsection{Function Space with Finite Expressions in FEX}
\label{sec:elementaryfun}
Mathematical expressions will be introduced to form the function space in FEX. 

\begin{definition}[Mathematical expression] A mathematical expression is a combination of symbols, which is well-formed by syntax and rules and forms a valid function. The symbols include operands (variables and numbers), operators (e.g.,``$+$'', ``$\sin$'', integral, derivative), brackets, punctuation.
\end{definition}

In the definition of mathematical expressions, we only consider the expression that forms a valid function. In our context, ``$\sin(x \times y)+1$'' is a mathematical expression, but, for example, ``$5>x$'' and ``$\sin(x \times y)+$'' are not a mathematical expression as they do not form a valid function. The operands and operators comprise the structure of an expression. The parentheses play a role in clarifying the operation order. 
\begin{definition}[$k$-finite expression] A mathematical expression is called a $k$-finite expression if the number of operators in this expression is $k$.
\end{definition}

``$\sin(x \times y)+1$'' is a $3$-finite expression since there are three operators in it (i.e., ``$\times$'', ``$\sin$'', and ``+''). The series, such as ``$1+x^1+\frac{x^2}{2}+\frac{x^3}{6}+\cdots$'', belongs to a mathematical expression, but it is not a finite expression since the amount of the operators is infinite. Formally, with the concept of finite expression, we can define FEX as follows,
\begin{definition}[Finite expression method] The finite expression method is a methodology to solve a PDE numerically by seeking a finite expression such that the resulting function solves the PDE approximately. 
\end{definition}

We denote $\mathbb{S}_k$ as a set of functions that are formed by finite expressions with the number of operators not larger than $k$. This $\mathbb{S}_k$ forms the function space in FEX. Clearly, $\mathbb{S}_1\subset \mathbb{S}_2\subset \mathbb{S}_3\cdots $. 

\subsection{Identifying PDE Solutions in FEX}
\label{sec:error} We denote a functional $\mathcal{L}: \mathbb{S} \to \mathbb{R}$ associated with a given PDE, where $\mathbb{S}$ is a function space and the minimizer of $\mathcal{L}$ is the best solution to solve the PDE in $\mathbb{S}$. In FEX, given the number of operators $k$, the problem of seeking a finite expression solution is formulated as a CO over $\mathbb{S}_k$ via 
\begin{align}\min \{\mathcal{L}(u)|u \in \mathbb{S}_k\}.\label{eqn:co}\end{align}

The choice of the functional $\mathcal{L}$ is problem dependent and one may conceive a better functional for a specific PDE with a specific constraint or domain. Some popular choices include least-square methods~\citep{lagaris1998artificial,sirignano2018dgm,nn1994}, variation formulations~\citep{yu2018deep,CiCP-29-1365,chen2020friedrichs,zang2020weak} and so on.

\subsubsection{Least Square Method} Suppose that the PDE is given by
\begin{equation}\label{eqn:BVP}
\begin{split}
 \mathcal{D}u(\bm{x})=f(u(\bm{x}),\bm{x}),\text{~}\bm{x} \in \Omega, \qquad
 \mathcal{B}u(\bm{x})=g(\bm{x}), \text{~}\bm{x} \in \partial\Omega,
\end{split}
\end{equation}
where $\mathcal{D}$ is a differential operator, $f(u(\bm{x}),\bm{x})$ can be a nonlinear function in $u$, $\Omega$ is a bounded domain in $\mathbb{R}^d$, and $\mathcal{B}u=g$ characterizes the boundary condition (e.g., Dirichlet, Neumann and Robin~\citep{evans2010partial}). The least square method~\citep{lagaris1998artificial,sirignano2018dgm,nn1994} defines a straightforward functional to characterize the error of the estimated solution by 
\begin{equation}\label{eqn:loss}
\mathcal{L}(u):=\|\mathcal{D}u(\bm{x})-f(u,\bm{x})\|_{L^2(\Omega)}^2+\lambda\|\mathcal{B}u(\bm{x})-g(\bm{x})\|_{L^2(\partial\Omega)}^2,
\end{equation}
 where $\lambda$ is a positive coefficient to enforce the boundary constraint. 

\subsubsection{Variation Formulation} We next introduce the variation formulation, which is commonly used to identify numerical PDE solutions~\citep{yu2018deep,CiCP-29-1365}. As an example, consider an elliptic PDE with homogeneous Dirichlet boundary conditions. This PDE is:
\begin{align}
     -\Delta u(\bm{x}) + c(\bm{x})u(\bm{x})=f(\bm{x}),\text{~}\bm{x} \in \Omega, \qquad u(\bm{x})=0, \text{~}\bm{x} \in \partial\Omega,
     \label{eqn:elliptic}
\end{align}
where $c$ is a bounded function and $f\in L^2$. The solution $u$ to PDE \eqref{eqn:elliptic} minimizes the variation formulation $\frac{1}{2}\int_\Omega\|\nabla u\|^2+cu^2\text{d}\bm{x}-\int_\Omega fu\text{d}\bm{x}$. By incorporating the boundary condition penalty into this variation, we obtain the functional:
\begin{equation}\label{eqn:vari}
\mathcal{L}(u):=\frac{1}{2}\int_\Omega\|\nabla u\|^2+cu^2\text{d}\bm{x}-\int_\Omega fu\text{d}\bm{x}+\lambda\int_{\partial\Omega} u^2\text{d}\bm{x}.
\end{equation}

 An alternative variation formulation can be defined using test functions. Let $v \in H_0^1(\Omega)$ be a test function, where $H_0^1(\Omega)$ denotes the Sobolev space whose weak derivative is $L^2$ integrable with zero boundary values. The weak solution $u$ of Eqn.~\eqref{eqn:elliptic} is defined as the function that satisfies the bilinear equations: 
\begin{align}
\begin{split}
    a(u,v):= \int_\Omega \nabla u \nabla v+cuv - fv \text{d}\bm{x} = 0, \quad \forall v \in H_0^1(\Omega), \quad u(\bm{x})=0, \text{~}\bm{x} \in \partial\Omega,
\end{split}
\end{align}
where $a(u,v)$ is constructed by multiplying \eqref{eqn:elliptic} and $v$, and integration by parts. All derivatives of the solution function can be transferred to the test function through repeated integration by parts, yielding another bilinear forms \citep{chen2020friedrichs}. The weak solution can be reformulated as the solution to a saddle-point problem~\citep{zang2020weak}:
\begin{align}
    \min_{u \in H_0^1(\Omega)} \max_{v \in H_0^1(\Omega)} |a(u,v)|^2/\|v\|_{L^2(\Omega)}^2.
\end{align}
Then, the functional $\mathcal{L}$ identifying the PDE solution is:
\begin{equation}\label{eqn:weak}
\mathcal{L}(u):=\max_{v \in H_0^1(\Omega)} |a(u,v)|^2/\|v\|_{L^2(\Omega)}^2+\lambda\int_{\partial\Omega} u^2\text{d}\bm{x}.
\end{equation}

\subsection{Approximation Theory of Elementary Expressions in FEX}
\label{sec:theory}

The most important theoretical question in solving high-dimensional problems is whether or not a solver suffers from the curse of dimensionality. It will be shown that the function space of $k$-finite expressions, i.e., $\mathbb{S}_k$ in \eqref{eqn:co}, is a powerful function space that avoids the curse of dimensionality in approximating high-dimensional continuous functions, leveraging the recent development of advanced approximation theory of deep neural networks \citep{Shen2021DeepNA,Jiao2021DeepNN}. First of all, it can be proved that $\mathbb{S}_k$ is dense in $C([0,1]^d)$ for an arbitrary $d\in\mathbb{N}$ in the following theorem.

\begin{theorem}
\label{thm:main}
The function space $\mathbb{S}_k$, generated with operators including  ``$+$'', ``$-$'', ``$\times$'', ``$/$'',  ``$|\cdot|$'', ``sign$(\cdot)$'', and ``$\lfloor \cdot \rfloor$'', is dense in $C([a,b]^d)$ for arbitrary $a$, $b\in \mathbb{R}$ and $d\in \mathbb{N}$ if $k\geq \mathcal{O}(d^4)$.
\end{theorem}

Here ``$|\cdot|$'', ``sign$(\cdot)$'', and ``$\lfloor \cdot \rfloor$'' denote the absolute, sign and floor functions~\citep{Shen2021DeepNA}, respectively. The proof of Theorem \ref{thm:main} can be found in Appendix \ref{sec:proof}. The denseness of $\mathbb{S}_k$ means that the function space of $k$-finite expressions can approximate any $d$-dimensional continuous functions to any accuracy, while $k$ is only required to be $\mathcal{O}(d^4)$ independent of the approximation accuracy. The proof of Theorem \ref{thm:main} takes the advantage of operators ``sign$(\cdot)$'' and ``$\lfloor \cdot \rfloor$'', which might not be frequently used in mathematical expressions. If it is more practical to restrict the operator list to regular operators like ``$+$'', ``$\times$'', ``$\sin(\cdot)$'', exponential functions, and the rectified linear unit (ReLU), then it can be proved that $\mathbb{S}_k$ can approximate H{\"o}lder functions without the curse of dimensionality in the following theorem.

\begin{theorem}
\label{thm:main2}
Suppose the function space $\mathbb{S}_k$ is generated with operators including  ``$+$'', ``$-$'', ``$\times$'', ``$\div$'',  ``$\max\{0,x\}$'', ``$\sin(x)$'', and ``$2^x$''. Let $p\in[1,+\infty)$. For any $f$ in the H{\"o}lder function class $\mathcal{H}^\alpha_\mu([0,1]^d)$ and $\varepsilon>0$, there exists a $k$-finite expression $\phi$ in $\mathbb{S}_k$ such that $\|f-\phi\|_{L^p}\leq \varepsilon$, if $k\geq \mathcal{O}(d^2(\log d+\log\frac{1}{\varepsilon})^2)$.
\end{theorem}

The proof of Theorem \ref{thm:main2} can be found in Appendix \ref{sec:proof}. Although finite expressions have a powerful approximation capacity for high-dimensional functions, it is challenging to theoretically prove that our FEX solver to be introduced in Section \ref{sec:alg} can identify the desired finite expressions with this power. Furthermore, the information-theoretic limitations discussed in~\citep{yarotsky2020phase} highlight that deep models with fewer weights would require significantly higher weight precision to achieve their theoretical approximation power. Similarly, for FEX, while the method may require only a few operators to form an explicit expression to approximate a high-dimensional function, it would require high precision to learn the parameters of these expressions. Given the finite precision in practical implementations, achieving the full theoretical approximation power of FEX for general high-dimensional functions can be challenging. Nevertheless, the numerical experiments in Section \ref{sec:numericalexp} demonstrate that our FEX solver can identify the desired finite expressions to machine precision for several classes of high-dimensional PDEs. Therefore, FEX would be an appealling alternative of existing tools for high-dimensional problems. More importantly, it is intuitive to understand when FEX can find solutions of high accuracy, solutions with explicit expressions, which is a class of functions widely adopted in benchmark comparisons in the literature, while deep models struggle to achieve high accuracy.

\section{An Implementation of FEX} 
\label{sec:alg}

Following the introduction of the abstract FEX methodology in Section \ref{sec:abstractframework}, this section proposes a numerical implementation. First, binary trees are applied to construct finite expressions in Section \ref{sec:elementary}. Next, our CO problem \eqref{eqn:co} is formulated in terms of parameter and operator selection to find expressions that approximate PDE solutions in Section \ref{sec:farmework}. To resolve this CO, we propose implementing a search loop to identify effective operators that have the potential to recover the true solution when selected for expression. In Appendix~\ref{appendix:code}, we provide pseudo-code for the FEX algorithm, which employs expanding trees to search for a solution. For the reader's convenience, we have summarized the key notations used in this section in Table~\ref{tab:notation}. 

\begin{table}[htbp]
  \centering
    \begin{adjustbox}{width=0.6\textwidth,center}
    \begin{tabular}{ll}
    \toprule
    \textbf{Notation} & \textbf{Explanation} \\
    \midrule
    $\mathcal{T}$     & A binary tree \\
    $\Be$     & Operator sequence  \\
    $\bm{\theta}$     & Trainable scaling and bias parameters\\
    $\mathcal{L}$     & The functional associated with the PDE solution \\
    $S$    & The scoring function that maps an operator sequence to $[0,1]$ \\
    $\bm{\chi}_\Phi$ & The controller parameterized by $\Phi$\\
    $\mathcal{J}$ & The objective function for the policy-gradient approach\\
    \bottomrule
    \end{tabular}%
  \end{adjustbox}
  \caption{A summary of notations in the FEX implementation.}
  \label{tab:notation}%
\end{table}%


\begin{figure}[ht]
		\centering
		\includegraphics[width=0.99\linewidth]{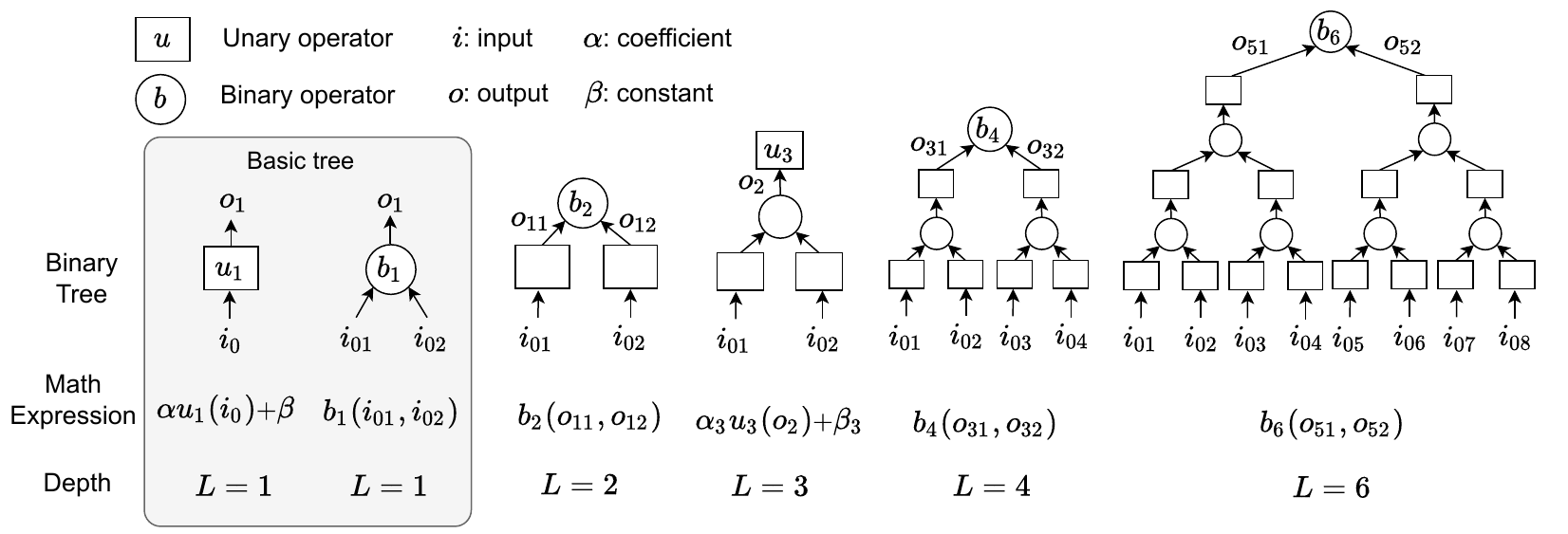}
		\caption{Computational rule of a binary tree. In a binary tree, each node holds either a unary or a binary operator. We begin by defining the computational process for a depth-1 tree (i.e., $L=1$), which contains only a single operator. For binary trees with more than depth one (i.e., $L>1$), the computation is performed recursively.}
		\label{fig:trees}
	\end{figure}

\subsection{Finite Expressions with Binary Trees}
\label{sec:elementary}
A finite expression can be represented as a binary tree $\mathcal{T}$, as depicted in Figure~\ref{fig:trees}. Each node in the tree is assigned a value from a set of operators, and all these node values collectively form an operator sequence $\Be$, following a predefined order to traverse the tree (e.g., inorder traversal). 

 Within each node featuring a unary operator, we incorporate two additional parameters, a scaling parameter $\alpha$ and a bias parameter $\beta$, to enhance expressiveness. All these parameters are denoted by $\bm{\theta}$. Therefore, a finite expression can be denoted as $u(\bm{x}; \mathcal{T}, \mathbf{e}, \bm{\theta})$, a function of $\bm{x}$. For a fixed $\mathcal{T}$, the maximal number of operators is upper bounded by a constant denoted $k_\mathcal{T}$. In this implementation, $\{u(\bm{x}; \mathcal{T}, \mathbf{e}, \bm{\theta}) |\mathbf{e}, \bm{\theta}\}$  forms the function space in the CO to solve a PDE. This is a subset of functions expressible with at most $k_\mathcal{T}$ finite expressions.

The configuration of the tree of various depths can be designed as in  Figure~\ref{fig:trees}. Each tree node is either a binary operator or a unary operator that takes value from the corresponding binary or unary set. The binary set can be $\mathbb{B}:=\{+,-,\times,\div,\cdots\}$. The unary set can be $\mathbb{U}:=\{\sin,\exp, \log, \text{Id}, (\cdot)^2, \int\cdot\text{d} x_i, \frac{\partial\cdot}{\partial x_i}, \cdots\}$, which contains elementary functions (e.g., polynomial and trigonometric function), antiderivative and differentiation operators. Here ``Id'' denotes the identity map. Note that if an integration or a derivative is used in the expression, the operator can be applied numerically. Each entry in the operator sequence $\Be$ has a one-to-one correspondence with the traversal of the nodes of the tree. The length of $\Be$ equals the total number of tree nodes. For example, when inorder traversal is used, Figure \ref{fig:workflow}b depicts a tree with 4 nodes and $\Be=(\text{Id}, \times, \exp, \sin).$

The computation flow of the binary tree is conducted recursively. The operator of a node is granted higher precedence than that of its parent node. First, as in Figure~\ref{fig:trees}, we present the computation flow of the basic trees (a tree has a depth of 1 with only 1 operator). For a basic tree with a unary operator $u_1$, when the input is $i_0$, then the output $o_1=\alpha u_1(i_0)+\beta$, where $\alpha$ and $\beta$ are scaling and bias parameters, respectively. For a basic tree with a binary operator $b_1$, when the input is $i_{01}$ and $i_{02}$, the output becomes $o_1=b_1(i_{01}, i_{02})$. With these two basic trees, we are ready to define the computation for arbitrary depth by recursion, as the examples shown in Figure~\ref{fig:trees}. Specifically, the input of a parent node is the output of the child node(s). When the tree input at the bottom layer is a $d$-dimensional variable $\bm{x}$ (Figure~\ref{fig:workflow}b), the unary operator directly linked to $\bm{x}$ is applied element-wisely to each entry of $\bm{x}$ and the scaling $\alpha$ becomes a $d$-dimensional vector to perform a linear transformation from $\mathbb{R}^d$ to $\mathbb{R}^1$. 

\subsection{Solving a CO in FEX}
\label{sec:farmework}
Given a tree $\mathcal{T}$, we aim to seek the PDE solution from the function space $\{u(\bm{x}; \mathcal{T}, \mathbf{e}, \bm{\theta}) |\mathbf{e}, \bm{\theta}\}\subset \mathbb{S}_{k_{\mathcal{T}}}$. The mathematical expression can be identified via the minimization of the functional $\mathcal{L}$ associated with a PDE, i.e., 
\begin{align}
\min \{\mathcal{L}(u(\cdot; \mathcal{T}, \Be, \bm{\theta}))|\Be, \bm{\theta}\}.
\label{eqn:obj}
\end{align}

We introduce the framework for implementing FEX, as displayed in Figure~\ref{fig:workflow}a, to seek a minimizer of (\ref{eqn:obj}). The basic idea is to find a good operator sequence $\Be$ that may uncover the structure of the true solution, and then optimize the parameter $\bm{\theta}$ to minimize the functional~\eqref{eqn:obj}. In our framework, the searching loop consists of four parts: 1) Score computation (i.e., rewards in RL). A mix-order optimization algorithm is proposed to efficiently assess the score of the operator sequence $\Be$ to uncover the true structure. A higher score suggests a higher possibility to help to identify the true solution. 2) Operator sequence generation (i.e., taking actions in RL). A controller is proposed to generate operator sequences with high scores (see Figure~\ref{fig:workflow}b). 3) Controller update (i.e., policy optimization in RL). The controller is updated to increase the probability of producing a good operator sequence via the score feedback of the generated ones. While the controller can be modeled in many ways (e.g., heuristic algorithm), we introduce the policy gradient in RL to optimize the controller. 4) Candidate optimization (i.e., a non-greedy strategy). During searching, we maintain a candidate pool to store the operator sequence with a high score. After searching, the parameters $\bm{\theta}$ of high-score operator sequences are optimized to approximate the PDE solution. 

\subsubsection{Score Computation} \label{sec:score}
The score of an operator sequence $\Be$ is critical to guide the controller toward generating good operator sequences and help to maintain a candidate pool of high scores. Intuitively, the score of $\Be$ is defined in the range $[0,1]$, namely $S(\Be)$, by 
\begin{align}
    S(\Be) := \big(1+L(\Be)\big)^{-1}, 
    \label{eqn:orgscore}
\end{align}
where $L(\Be):=\min \{\mathcal{L}(u(\cdot; \mathcal{T}, \Be, \bm{\theta}))|\bm{\theta}\})$. When $L(\Be)$ tends to $0$, the expression represented by $\Be$ is close to the true solution, and the score $S(\Be)$ goes to 1. Otherwise, $S(\Be)$ goes to 0. The global minimizer of $\mathcal{L}(u(\cdot; \mathcal{T}, \Be, \bm{\theta}))$ over $\bm{\theta}$ is difficult and expensive to obtain. Instead of exhaustively searching for a global minimizer, a first-order optimization algorithm and a second-order one are combined to accelerate the evaluation of $S(\Be)$.

First-order algorithms (e.g., the stochastic gradient descent~\citep{rumelhart1986learning} and Adam~\citep{kingma2014adam}) that utilize gradient to update are popular in machine learning. Typically, they demand a small learning rate and a substantial number of iterations to converge effectively. It can become time-consuming to optimize $L(\Be)$ using a first-order algorithm.  Alternatively, second-order algorithms (e.g., the Newton method~\citep{avriel2003nonlinear} and the Broyden-Fletcher-Goldfarb-Shanno method (BFGS)~\citep{fletcher2013practical}) use the (approximated) Hessian matrix for faster convergence, but obtaining a good minimizer requires a good initial guess. To expedite the optimization process of $L(\Be)$ in our implementation, we employ a two-step approach. Initially, a first-order algorithm is utilized for $T_1$ steps to obtain a well-informed initial guess. Subsequently, a second-order algorithm is applied for an additional $T_2$ steps to further refine the solution. Let $\bm{\theta}_0^{\Be}$ be an initialization and $\bm{\theta}_{T_1+T_2}^{\Be}$ be the parameter set after $T_1+T_2$ steps of this two-step optimization. Then $\bm{\theta}_{T_1+T_2}^{\Be}$ serves as an approximate of $\arg \min_{\bm{\theta}} \mathcal{L}(u(\cdot; \mathcal{T}, \Be, \bm{\theta}))$. Finally, $S(\Be)$ is estimated by 
\begin{align}
    S(\Be) \approx \big(1+\mathcal{L} (u(\cdot; \mathcal{T}, \Be, \bm{\theta}_{T_1+T_2}^{\Be}))\big)^{-1}. 
    \label{eqn:score}
\end{align}
Remark that the approximation may exhibit significant variation due to the randomness associated with the initialization of $\bm{\theta}_0^{\Be}$.

\subsubsection{Operator Sequence Generation} The role of the controller is to generate operator sequences with high scores during the searching loop. Let $\bm{\chi}_\Phi$ be a controller with model parameter $\Phi$, and $\Phi$ is updated to increase the probability for good operator sequences during the searching loop. We use $\Be\sim\bm{\chi}_\Phi$ to denote the process to sample an $\Be$ according to the controller $\bm{\chi}_\Phi$. 

Treating tree node values of $\mathcal{T}$ as random variables, the controller $\bm{\chi}_\Phi$ outputs probability mass functions $\Bp_\Phi^1, \Bp_\Phi^2, \cdots, \Bp_\Phi^s$ to characterize their distributions, where $s$ is the total number of nodes. Each tree node value $e_j$ is sampled from $\Bp_\Phi^j$ to obtain an operator. Then $\Be:=(e_1, e_2, \cdots, e_s)$ is the operator sequence sampled from $\bm{\chi}_\Phi$. See Figure~\ref{fig:workflow}b for an illustration. Besides, we adopt the $\epsilon$-greedy strategy~\citep{sutton2018reinforcement} to enhance exploration of a potentially high-score $\Be$. With probability $\epsilon<1$, $e_i$ is sampled from a uniform distribution of the operator set. With probability $1-\epsilon$, $e_i\sim \Bp_\Phi^i$. A larger $\epsilon$ leads to a higher probability of exploring new sequences. 

\subsubsection{Controller Update} The goal of the controller update is to guide the controller toward generating high-score operator sequences $\Be$. The updating rule of a controller can be designed based on heuristics (e.g., genetic and simulated annealing algorithms) and gradient-based methods (e.g., policy gradient and darts~\citep{liu2018darts}). As proof of concept, we introduce a policy-gradient-based updating rule in RL. The policy gradient method aims to maximize the return by optimizing a parameterized policy, and the controller in our problem plays the role of a policy. 

In this paper, the controller $\bm{\chi}_\Phi$ is modeled as a neural network parameterized by $\Phi$. The training objective of the controller is to maximize the expected score of a sampled $\Be$, i.e., 
\begin{align}
\mathcal{J}(\Phi):=\mathbb{E}_{\Be \sim \bm{\chi}_\Phi} S(\Be).
\label{eqn:expect}
\end{align}
Taking the derivative of \eqref{eqn:expect} with respect to $\Phi$, we have 
\begin{align}
\nabla_\Phi\mathcal{J}(\Phi)=\mathbb{E}_{\Be \sim \bm{\chi}_\Phi} \big\{S(\Be)\sum_{i=1}^s \nabla_\Phi \log(\Bp_\Phi^i(e_i))\big\},
\label{eqn:expectgrad}
\end{align}
where $\Bp_\Phi^i(e_i)$ is the probability corresponding to the sampled $e_i$. When the batch size is $N$ and $\{\Be^{(1)}, \Be^{(2)}, \cdots, \Be^{(N)}\}$ are sampled under $\bm{\chi}_\Phi$ each time, the expectation can be approximated by 
\begin{align}
\nabla_\Phi\mathcal{J}(\Phi)\approx \frac{1}{N}\sum_{k=1}^N \big\{S(\Be^{(k)})\sum_{i=1}^s \nabla_\Phi \log(\Bp_\Phi^i(e_i^{(k)}))\big\}.
\label{eqn:mcmcavg}
\end{align}
Next, the model parameter $\Phi$ is updated via the gradient ascent with a learning rate $\eta$, i.e., 
\begin{align}
\Phi \leftarrow \Phi+\eta \nabla_\Phi\mathcal{J}(\Phi). 
\label{eqn:gradientascent}
\end{align}
The objective in \eqref{eqn:expect} helps to improve the average score of generated sequences. In our problem, the goal is to find $\Be$ with the best score. To increase the probability of obtaining the best case, the objective function proposed in \citep{petersen2021deep} is applied to seek the optimal solution via
\begin{align}
\mathcal{J}(\Phi)=\mathbb{E}_{\Be \sim \bm{\chi}_\Phi} \{S(\Be)|S(\Be)\geq S_{\nu, \Phi}\},
\label{eqn:expectriskseeking}
\end{align}
where $S_{\nu, \Phi}$ represents the $(1-\nu)\times 100\%$-quantile of the score distribution generated by $\bm{\chi}_\Phi$. In a discrete form, the gradient computation becomes 
\begin{align}
\nabla_\Phi\mathcal{J}(\Phi)\approx \frac{1}{N}\sum_{k=1}^N \big\{(S(\Be^{(k)})-\hat{S}_{\nu, \Phi})\mathbbm{1}_{\{S(\Be^{(k)})\geq \hat{S}_{\nu, \Phi}\}}\sum_{i=1}^s \nabla_\Phi \log(\Bp_\Phi^i(e_i^{(k)}))\big\}.
\label{eqn:mcmcrisk}
\end{align}
where $\mathbbm{1}$ is an indicator function that takes value $1$ if the condition is true otherwise 0, and $\hat{S}_{\nu, \Phi}$ is the $(1-\nu)$-quantile of the scores $\{S(\Be^{(i)})\}_{i=1}^N$.

\subsubsection{Candidate Optimization} As introduced in Section~\ref{sec:score}, the score of $\Be$ is based on the optimization of a nonconvex function at a random initialization. Therefore, the optimization may get stuck at poor local minimizers and the score sometimes may not reflect whether $\Be$ reveals the structure of the true solution. The operator sequence $\Be$ corresponding to the true solution (or approximately) may not have the best score. For the purpose of not missing good operator sequences, a candidate pool $\mathbb{P}$ with capacity $K$ is maintained to store several $\Be$'s of a high score. 

During the search loop, if the size of $\mathbb{P}$ is less than the capacity $K$, $\Be$ will be put in $\mathbb{P}$. If the size of $\mathbb{P}$ reaches $K$ and $S(\Be)$ is larger than the smallest score in $\mathbb{P}$, then $\Be$ will be appended to $\mathbb{P}$ and the one with the least score will be removed. After the searching loop, for every $\Be\in \mathbb{P}$, the objective function $\mathcal{L}(u(\cdot; \mathcal{T}, \Be, \bm{\theta}))$ is optimized over $\bm{\theta}$ using a first-order algorithm with a small learning rate for $T_3$ iterations. 

\section{Numerical Examples}
\label{sec:numericalexp}
Numerical results will be provided to demonstrate the effectiveness of our FEX implementation introduced in Section~\ref{sec:farmework} using two classical PDE problems: high-dimensional PDEs with constraints (such as Dirichlet boundary conditions and integration constraints) and eigenvalue problems. The computational tools for high-dimensional problems are very limited and NNs are probably the most popular ones. Therefore, FEX will be compared with NN-based solvers. Through our examples, the goal is to numerically demonstrate that:
\begin{itemize}
\item {\bf Accuracy.} FEX can achieve high and even machine accuracy for high-dimensional problems, while NN-based solvers can only achieve the accuracy of $\mathcal{O}(10^{-4})$ to $\mathcal{O}(10^{-2})$. Furthermore, we demonstrate the effectiveness of our RL-based approach by comparing its performance to a solution developed using genetic programming (GP)~\citep{murray2012modelling}.

\item {\bf Scalability.} FEX is scalable in the problem dimension with an almost constant accuracy and a low memory requirement, i.e., the accuracy of FEX remains essentially the same when the dimension grows, while NN-based solvers have a worse accuracy when the dimension becomes larger. 

\item {\bf Interpretability.} FEX provides interpretable insights of the ground truth PDE solution and helps to design postprocessing techniques for a refined solution. 
\end{itemize}
In particular, to show the benefit of interpretability, we will provide examples to show that the explicit formulas of FEX solutions help to design better NN-parametrization in NN-based solvers to achieve higher accuracy. The FEX-aided NN-based solvers are referred to as FEX-NN in this paper. 
Finally, we will show the convergence of FEX with the growth of the tree size when the true solution can not be exactly reproduced by finite expressions using the available operators and a binary tree. All results of this section are obtained with $6$ independent experiments to achieve their statistics. 

\subsection{Experimental Setting}\label{sec:numerical} This part provides the setting of FEX and NN-based solvers. The accuracy of a numerical solution $\tilde{u}$  compared with the true solution $u$ is measured by a \textit{relative $L^2$ error}, i.e.,  $\|\tilde{u}-u\|_{L^2(\Omega)}/\|u\|_{L^2(\Omega)}$. The integral in the $L^2$ norm is estimated by the Monte Carlo integral for high-dimensional problems.

\textbf{Implements of FEX.} The depth-3 binary tree (Figure~\ref{fig:workflow}b) with 3 unary operators and 1 binary operator is used to generate mathematical expressions. The binary set is $\mathbbm{B}=\{+,-,\times\}$ and the unary set is $\mathbbm{U}=\{0, 1, \text{Id}, (\cdot)^2, (\cdot)^3, (\cdot)^4, \exp, \sin, \cos\}$. A fully connected NN is used as a controller $\bm{\chi}_\Phi$ with constant input (see Appendix~\ref{appendix:controller} for more details). The output size of the controller NN is $n_1|\mathbbm{B}|+n_2|\mathbbm{U}|$, where $n_1=1$ and $n_2=3$ represent the number of binary and unary operators, respectively, and $|\cdot|$ denotes the cardinality of a set.

There are four main parts in the implementation of FEX as introduced in Section \ref{sec:farmework}. We will only briefly describe the key numerical choices here. (1) \textit{Score computation.} The score is updated first by Adam with a learning rate 0.001 for $T_1=20$ iterations and then by BFGS with a learning rate 1 for maximum $T_2=20$ iterations. (2) \textit{Operator sequence generation.} The depth-3 binary tree (Figure~\ref{fig:workflow}b) with 3 unary operators and 1 binary operator is used to generate mathematical expressions. The binary set is $\mathbbm{B}=\{+,-,\times\}$ and the unary set is $\mathbbm{U}=\{0, 1, \text{Id}, (\cdot)^2, (\cdot)^3, (\cdot)^4, \exp, \sin, \cos\}$. A fully connected NN is used as a controller $\bm{\chi}_\Phi$ with constant input. The output size of the controller NN is $n_1|\mathbbm{B}|+n_2|\mathbbm{U}|$, where $n_1=1$ and $n_2=3$ represent the number of binary and unary operators, respectively, and $|\cdot|$ denotes the cardinality of a set. (3) \textit{Controller update.} The batch size for the policy gradient update is $N=10$ and the controller is trained for $1000$ iterations using Adam with a fixed learning rate $0.002$. Especially in Section of ``Numerical Convergence'', since the deeper trees are used, the controller is updated for 5000 iterations for trees with different depth. We adopt the $\epsilon$-greedy strategy to increase the exploration of new $\Be$. The probability $\epsilon$ of sampling an $e_i$ by random is $0.1$. (4) \textit{Candidate optimization.} The candidate pool capacity is set to be $K=10$. For any $\Be\in\mathbb{P}$, the parameter $\bm{\theta}$ is optimized using Adam with an initial learning rate $0.01$ for $T_3=20,000$ iterations. The learning rate decays according to the cosine decay schedule~\citep{he2019bag}. 

\textbf{Implements of NN-based Solvers.} Residual networks (ResNets)~\citep{he2016deep,yu2018deep} $u(\bm{x};\bm{\Theta})$ parameterized by $\bm{\Theta}$ are used to approximate a solution and a minimization problem ${\min}_{\bm{\Theta}}~ \mathcal{L}(u(\cdot;\bm{\Theta}))$ is solved to identify a numerical solution~\citep{sirignano2018dgm,weinan2021algorithms,yu2018deep}.

The ResNet maps from $\mathbb{R}^d$ to $\mathbb{R}^1$ and consists of seven fully connected layers with three skip connections. Each hidden layer contains 50 neurons. The neural network is optimized using the Adam optimizer with an initial learning rate of 0.001 for 15,000 iterations. The learning rate is decayed following a cosine decay schedule.

\textit{Poisson equation.} The coefficient $\lambda$ in the functional~\eqref{eqn:loss} is 100. The batch size for the interior and boundary is 5,000 and 1,000, respectively. In the NN method, we use the ReLU$^2$ activation, i.e., $(\max\{x, 0\})^2$, in ResNet to approximate the solution.  

\textit{Linear conservative law.} The coefficient $\lambda$ in the functional~\eqref{eqn:conservative} is 100. In the NN method, we use ReLU $(\max\{x, 0\})$ activation. The batch size for the interior and boundary is 5,000 and 1,000, respectively. We use the same batch size as the NN method except that we increase the batch size to 20,000 for the interior and 4,000 for the boundary when the dimension is not smaller than 36. 

\textit{Schr\"odinger equation.} The coefficient $\lambda$ in the functional~\eqref{eqn:shrodinger} is 1. The batch size for estimating the first term and second term of~\eqref{eqn:shrodinger} is 2,000 and 10,000, respectively. Besides, ReLU$^2$ is used in ResNet.

\textbf{Implements of genetic programming (GP).} GP is an evolutionary algorithm used for automated symbolic regression and the evolution of computer programs to solve complex problems. Through a process of selection, crossover, and mutation, GP evolves and refines these programs over generations to optimize a specific fitness or objective function. We utilize GP to optimize the CO~\eqref{eqn:co} directly using the Python package GPlearn~\citep{gplearn}. 

\subsection{High-dimensional PDEs} Several numerical examples for high-dimensional PDEs including linear and nonlinear cases are provided here. In these tests, the true solutions of PDEs have explicit formulas that can be reproduced by the binary tree defined in Section~\ref{sec:elementary} and $\mathbbm{B}, \mathbbm{U}$ defined in Section~\ref{sec:numerical}.

\begin{figure}[ht]
		\centering
        \includegraphics[width=0.80\linewidth]{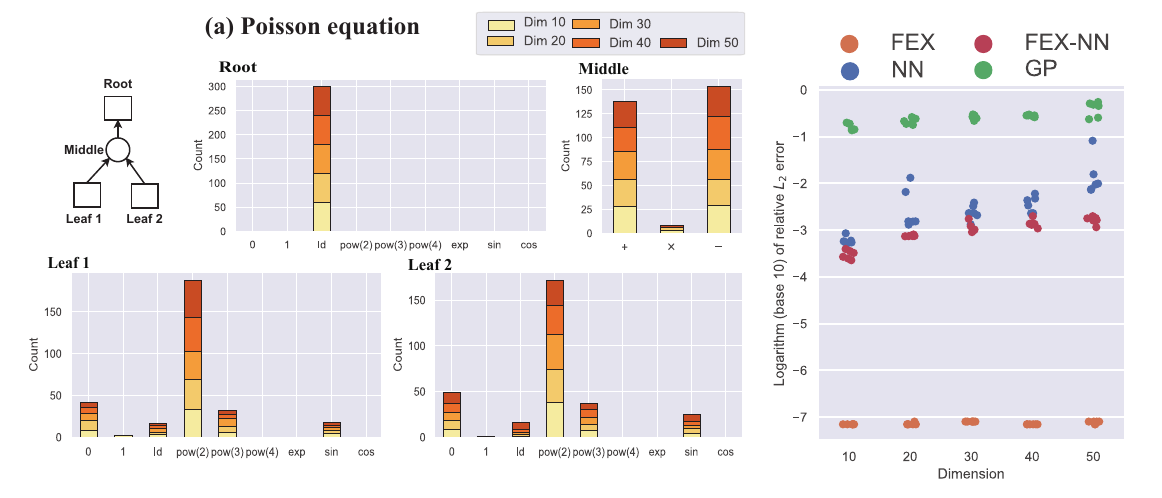}
        \includegraphics[width=0.80\linewidth]{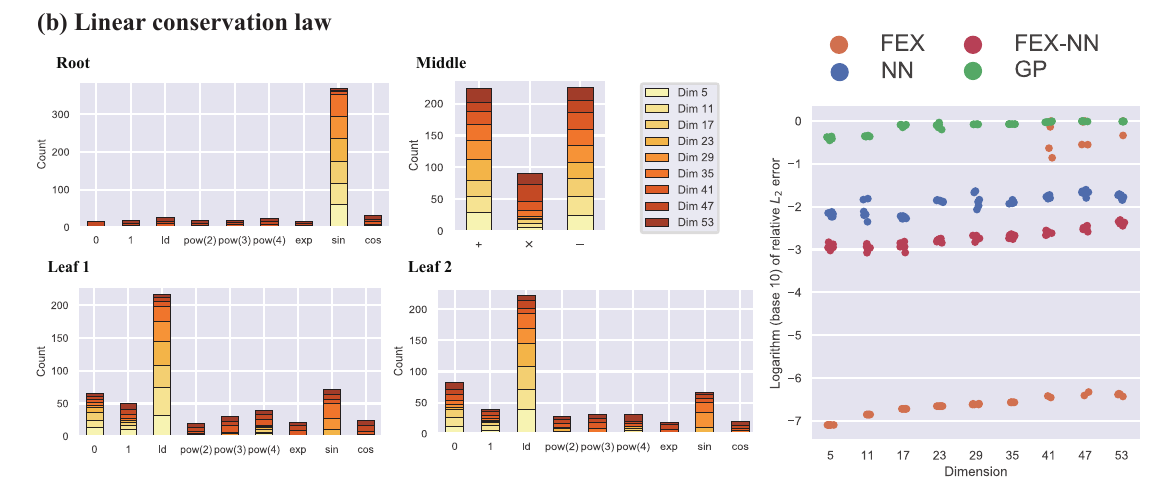}
        \includegraphics[width=0.80\linewidth]{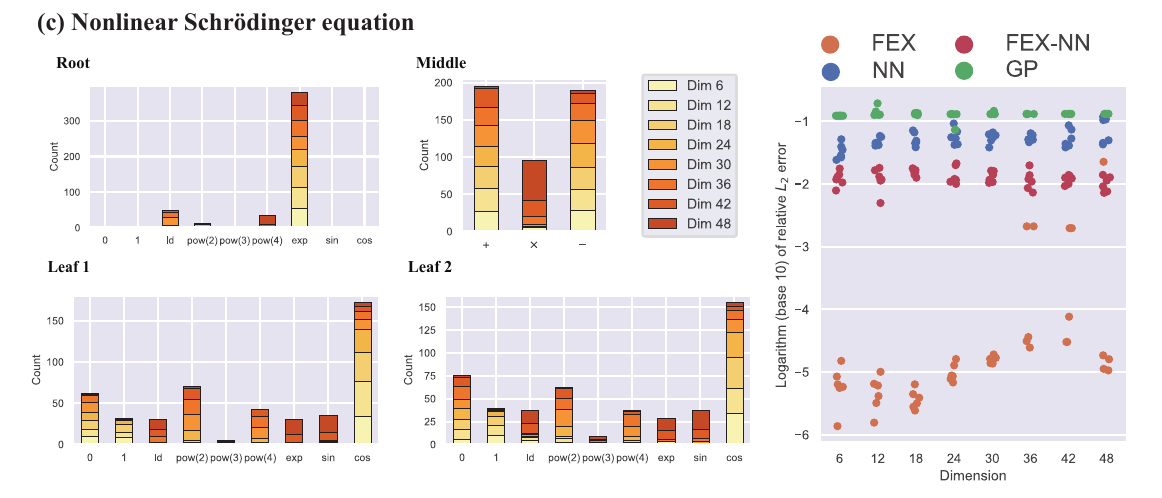}
		\caption{Distribution of the node values and the error comparison. In FEX, we  search the optimal sequence of node values and show the frequency of the node values of the binary tree in the candidate pool, consisting of the root (Root), middle node (Middle), and two leaves (Leaf 1 and Leaf 2). Based on the observation of the distribution, we readily design the new NN parameterization (FEX-NN) to estimate the solution. The last column displays comparison of the relative $L^2$ error as the function of the dimension over 6 independent trials between FEX, NN, FEX-NN and Genetic Programming (GP) for various high-dimensional PDE problems. Rows (a), (b) and (c) represent the results for Poisson equation~\eqref{eqn:poisson}, Linear conservation law~\eqref{eqn:conservationlaw1} and Nonlinear Schr\"odinger equation~\eqref{equ:schrodinger} respectively. For various dimensions, FEX identifies the true solution, approximating solutions of almost the machine accuracy. }
		\label{fig:depth3_tree_results}
\end{figure}

\subsubsection{Poisson Equation} 
We consider a Poisson equation~\citep{yu2018deep} with a Dirichlet boundary condition on a $d$-dimensional domain $\Omega=[-1,1]^d$, 
\begin{equation}
\label{eqn:poisson}
-\Delta u = f \text{ for }  \Bx \in \Omega, \quad u=g \text{ for }  \Bx \in \partial\Omega.
\end{equation}
Let the true solution be $\frac{1}{2}\sum_{i=1}^d x_i^2$, and then $f$ becomes a constant function $-d$. The functional $\mathcal{L}$ defined by least-square error \eqref{eqn:loss} is applied in the NN-based solver and FEX to seek the PDE solution for various dimensions ($d=10$, $20$, $30$, $40$ and $50$). 
\begin{figure}[ht]
		\centering
		\includegraphics[width=0.9\linewidth]{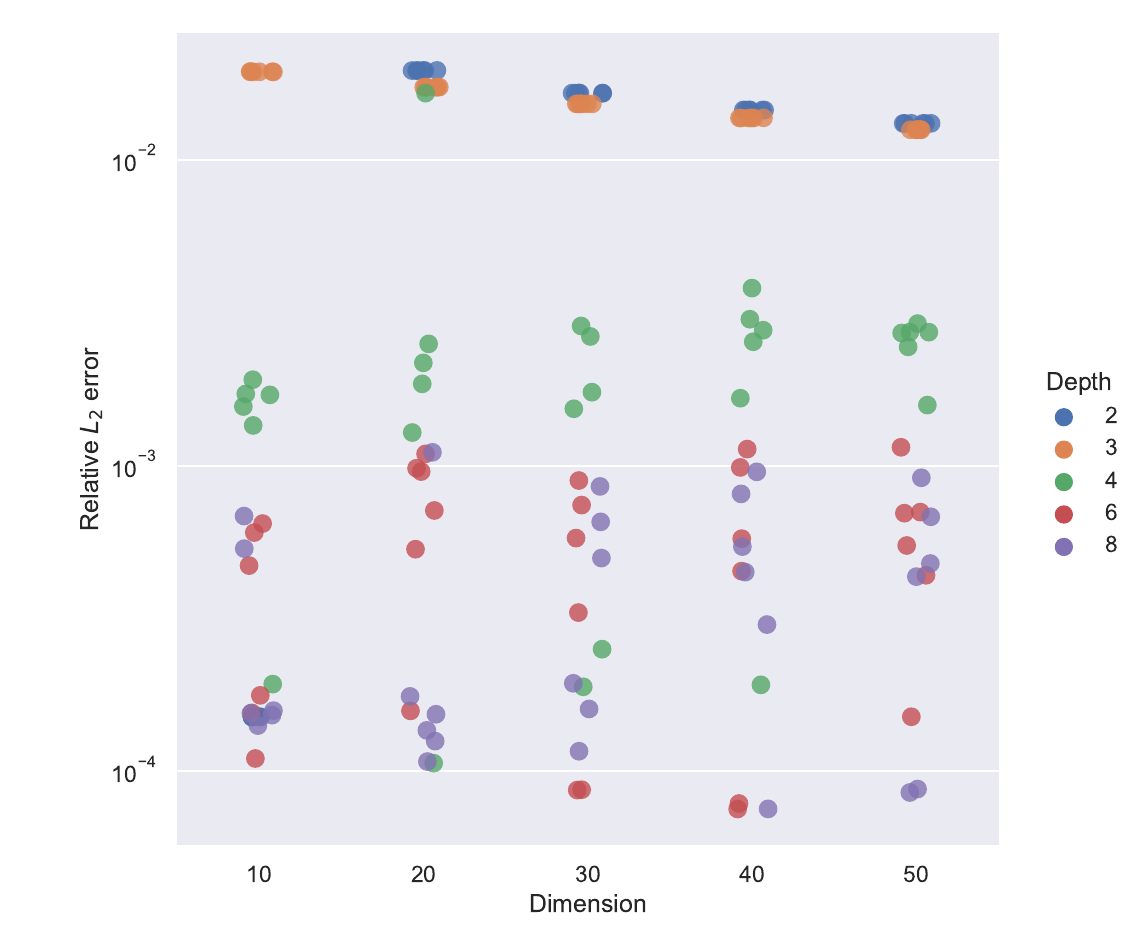}
		\caption{Relative $L_2$ error of solutions estimated by trees with increasing depth for the problems of various dimensions. We exclude the square operator $(\cdot)^2$ in the unary set $\mathbbm{U}$, and the binary tree defined in Section~\ref{sec:elementary} can not reproduce the true solution (sum of the square of coordinates) exactly in the example of the Poisson equation. We found that smaller errors could be obtained with larger tree sizes. }
		\label{fig:convergence}
\end{figure}

\subsubsection{Linear Conservation Law} The linear conservation law \citep{chen2021deep} is considered here with a domain $T\times\Omega=[0,1]\times[-1,1]^d$ and an initial value:

\begin{align}
\begin{split}
\label{eqn:conservationlaw1}
\frac{\pi d}{4}  u_t - \sum_{i=1}^d u_{x_i}= 0 \text{ for }  \Bx=(x_1,\cdots,x_d) \in \Omega, t\in [0,1], \quad u(0,\Bx) = \sin(\frac{\pi}{4} \sum_{i=1}^d x_i) \text{ for }  \Bx \in \Omega,
\end{split}
\end{align}
where the true solution is $u(t,\Bx) = \sin(t+\frac{\pi}{4}\sum_{i=1}^d x_i)$, and $d=5$, $11$, $17$, $23$, $29$, $35$, $41$, $47$ and $53$ in our tests. The functional $\mathcal{L}$ used in the NN-based solver and FEX to identify the solution is defined by 
\begin{align}
    \mathcal{L}(u):=\|\frac{\pi d}{4}  u_t - \sum_{i=1}^d u_{x_i}\|_{L^2(T\times\Omega)}^2+\lambda\|u(0,\Bx) - \sin(\frac{\pi}{4} \sum_{i=1}^d x_i)\|_{L^2(\Omega)}^2.
    \label{eqn:conservative}
\end{align}
\subsubsection{Nonlinear Schr\"odinger Equation}
We consider a nonlinear Schr\"odinger equation~\citep{han2020solving} with a cubic term on a $d$-dimensional domain $\Omega=[-1,1]^d$,
\begin{equation}
\label{equ:schrodinger}
	\begin{aligned}
		- \Delta u + u^3 + V u &= 0 \text{ for }  \Bx \in \Omega,
	\end{aligned}
\end{equation}  
where $V(\Bx) = -\frac{1}{9}\exp(\frac{2}{d}\sum_{i=1}^d\cos x_i)+ \sum_{i=1}^d(\frac{\sin^2 x_i}{d^2}- \frac{\cos x_i}{d})$ for $\Bx=(x_1,\cdots,x_d)$. And we let $\hat{u}(\Bx) = \exp( \frac{1}{d}\sum_{j=1}^d\cos(x_j ))/3$ be the solution of the PDE~\eqref{equ:schrodinger}. To avoid the trivial zero solution, we apply different strategies during the score computation and candidate optimization phases. During the score computation phase, the norm of the test function $u$, i.e., $\|u\|_{L_2(\Omega)}$, is used as a penalty for the function that is close to zero by 
\begin{align}
    \mathcal{L}_1(u):=\|- \Delta u + u^3 + V u  \|_{L_2(\Omega)}^2/\|u\|_{L_2(\Omega)}^3.
    \label{eqn:shrodinger0}
\end{align}

During the candidate optimization phase, an integration constraint is imposed to~\eqref{equ:schrodinger}, i.e., $\int_{\Omega} u(\Bx) d\Bx = \int_{\Omega} \hat{u}(\Bx) d\Bx$. The functional $\mathcal{L}$ used in FEX to fine-tune the identified operator sequence is defined by 
\begin{align}
    \mathcal{L}_2(u):=\|- \Delta u + u^3 + V u  \|_{L_2(\Omega)}^2+\lambda\big( \int_{\Omega} u(\Bx) d\Bx-\int_{\Omega} \hat{u}(\Bx) d\Bx\big)^2,
    \label{eqn:shrodinger}
\end{align}
where the second term imposes the integration constraint. The functional~\eqref{eqn:shrodinger} is also used in the NN-based solver to approximate the PDE solution.  
 Various dimensions are tested in the numerical results, e.g., $d=6$, $12$, $18$, $24$, $30$, $36$, $42$ and $48$. Remark that we avoid using~\eqref{eqn:shrodinger} in the score computation because the Monte-Carlo error tends to be significant in the second term of~\eqref{eqn:shrodinger} when the batch size is small, but using a large batch size can lead to inefficient computations. Hence, for the computational efficiency, we will utilize~\eqref{eqn:shrodinger0} instead in the score computation, rather than~\eqref{eqn:shrodinger}.

\subsection{Results}

Three main sets of numerical results for the PDE problems above will be presented. First, the errors of the numerical solutions by NN-based solvers and FEX are compared. Second, a convergence test is analyzed when the tree size of FEX increases. Finally, FEX is applied to design special NN parametrization to solve PDEs in NN-based solvers.

\textbf{Estimated Solution Error}. The depth-3 binary tree (Figure~\ref{fig:trees}) is used in FEX with four nodes (a root node (R), a middle node (M), and two leave nodes (L1 and L2)). Figure~\ref{fig:depth3_tree_results} shows the operator distribution obtained by FEX and the error comparison between the NN method and our FEX. The results show that NN solutions have numerical errors between $\mathcal{O}(10^{-4})$ and $\mathcal{O}(10^{-2})$ and the errors grow in the problem dimension $d$, agreeing with the numerical observation in the literature. Meanwhile, FEX can identify the true solution structure for the Poisson equation and the linear conservation law with errors of order $10^{-7}$, reaching the machine accuracy since the single-float precision is used. In the results of the nonlinear Schr\"odinger equation, FEX identifies the solutions of the form $\exp(\cos(\cdot))$ but achieves errors of order $10^{-5}$. 
Note that $\int_{\Omega} \hat{u}(\Bx) d\Bx$ in \eqref{eqn:shrodinger} is estimated by the Monte-Carlo integration with millions of points as an accurate and precomputed constant, but $\int_{\Omega} u(\Bx) d\Bx$ can only be estimated with fixed and small batch size, typically less than 10,000, in the optimization iterations. As the dimension grows, the estimation error of $\int_{\Omega} u(\Bx) d\Bx$ increases, and, hence, even the ground true solution has an increasingly large error according to \eqref{eqn:shrodinger}. Therefore, the optimization solver may return an approximate solution without machine accuracy. Designing a functional $\mathcal{L}$ free of the Monte-Carlo error (e.g., Eqns. \eqref{eqn:loss} and \eqref{eqn:conservative}) for the nonlinear Sch\"odinger equation could ensure machine accuracy.  
Furthermore, as shown in Figure~\ref{fig:depth3_tree_results}, GP tends to identify solutions of lower quality that fail to achieve a high level of accuracy. 

\textbf{Numerical Convergence.} The numerical convergence analysis is performed using the Poisson equation as an example. Binary trees of depths 2, 3, 4, 6 and 8 are used (see Figure~\ref{fig:trees}). The square operator $(\cdot)^2$ is excluded in $\mathbbm{U}$ so that the binary tree defined in Section~\ref{sec:elementary} can not reproduce the true solution (sum of the square of coordinates) exactly. This setting can mimic the case of a complicated solution while a small binary tree was used. Figure~\ref{fig:convergence} shows the error distribution of FEX with the growth of dimensions and the change of tree depths. FEX obtains smaller errors with increasing tree size. Notice that, compared with the errors of NN-based solvers reported in Figure~\ref{fig:depth3_tree_results}, FEX gets a higher accuracy when a larger tree is used. These results underscore FEX's ability to identify better solutions without relying on any special construction and provide evidence that its performance is not simply due to biases in the construction, such as tree structure or operator set.

\textbf{FEX-NN}. FEX provides interpretable insights of the ground truth PDE solution by the operator distribution obtained from the searching loop. It may be beneficial to design NN models with a special structure to increase the accuracy of NN-based solvers. In the results of the Poisson equation, we observe that the square operator has a high probability to appear at the leave nodes, which suggests that the true solution may contain the structure $\Bx^2:=(x_1^2, \cdots, x_d^2)$ at the input $\Bx$. As a result, we define the FEX-NN by $v(\Bx^2;\bm{\Theta})$ for the Poisson equation. Similarly, we use FEX-NNs $\sin(v(\Bx;\bm{\Theta}))$ for the linear conservation law and $\exp(v(\Bx^2;\bm{\Theta}))$ for the nonlinear Schr\"odinger equation. Figure~\ref{fig:depth3_tree_results} shows the errors of FEX-NN with the growth of dimensions, and it is clear that FEX-NN outperforms the vanilla NN-based method by a significant margin. 



\subsection{Eigenvalue Problem}
\label{sec:eigen}
Consider identifying the smallest eigenvalue $\gamma$ and the associated eigenfunction $u$ of the eigenvalue problem~\citep{yu2018deep}, 
\begin{align}
\label{eqn:eigenvalue}
\begin{split}
-\Delta u+ w u = \gamma u, \quad \Bx \in \Omega , \quad \text{for} \ 
u = 0 \ \text{on} \ \partial\Omega .
\end{split}
\end{align}
    The minimization of the Rayleigh quotient $\mathcal{I}(u)=\frac{\int_\Omega \|\nabla u\|_2^2d\Bx+\int_\Omega w  u^2 d\Bx}{\int_\Omega u^2 d\Bx}$, s.t., $u|_{\partial\Omega}=0$, gives the smallest eigenvalue and the corresponding eigenfunction. In the NN-based solver~\citep{yu2018deep}, the following functional is defined 
\begin{align}
\label{eqn:eigenvalueloss}
\begin{split}
\mathcal{L}(u):=\mathcal{I}(u)+\lambda_1 \int_{\partial \Omega} u^2 d\Bx + \lambda_2 \big(\int_{\Omega} u^2 d\Bx-1\big)^2
\end{split}
\end{align}
to seek an NN solution. 

\subsubsection{FEX vs NN for non-analytic eigenfunctions}
We consider a potential $w(\Bx)=\|\Bx\|_2^2+\delta \sum_{i=1}^d x_i^4$ and $\Omega=\mathbb{R}^d$, which incorporates a quartic perturbation term $\delta \sum_{i=1}^d x_i^4$ with a small $\delta\geq 0$. The domain $\Omega$ is truncated from $\mathbb{R}^d$ to $[-3,3]^d$ for simplification as done in~\citep{yu2018deep}. When $\delta=0$, the smallest eigenvalue of (\ref{eqn:eigenvalue}) is $d$ and the associated eigenfunction is $\exp(-\frac{\|\Bx\|_2^2}{2})$. Potentials with quartic terms (e.g., $\delta>0$) can lead to eigenfunctions in non-analytic form. While the eigenfunction corresponding to the smallest eigenvalue can be non-analytic, perturbation theory suggests that the true smallest eigenvalue can be approximated by $n+3/4n\delta$ (with details in Appendix~\ref{sec:eignref}), which we use as a reference for the true smallest eigenvalue. 

 We consider a small $\delta=0.1$. Both the FEX and NN methods use the functional~\eqref{eqn:eigenvalueloss} to estimate the eigenvalue. We used the same batch size across different dimensions to discretize the integration over the domain, as well as the integration over the boundary. Specifically, the batch size for estimating the first term and third term of the objective function~\eqref{eqn:eigenvalueloss} is
500,000 while that of the second term (boundary) is 100,000. Figure~\ref{fig:per_eigen} compares FEX and the NN method to approximate the smallest eigenvalue for $d=10,12,14,16$ and $18$ using the Rayleigh quotient~\eqref{eqn:eigenvalueloss}. As the dimension $d$ increases, the smallest eigenvalues identified by the NN method deviate significantly from the reference value, whereas those obtained using FEX remain much closer to it. Furthermore, the FEX method demonstrates greater stability, with substantially smaller standard deviation in the identified smallest eigenvalues.

\begin{figure}[h]
		\centering
		\includegraphics[width=0.9\linewidth]{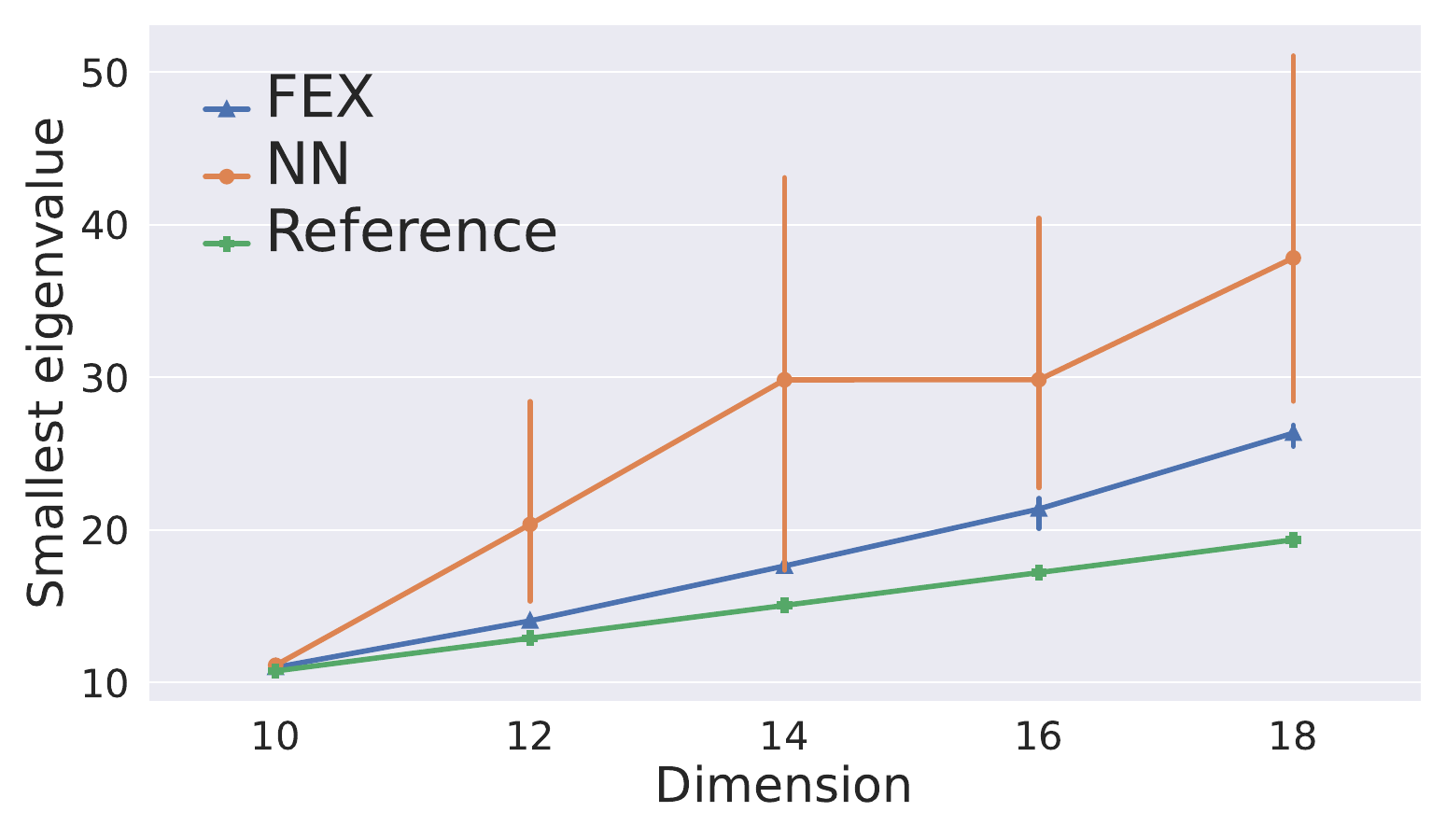}
		\caption{Comparison of the FEX method and the NN approach for approximating the smallest eigenvalue of \eqref{eqn:eigenvalue} using the Rayleigh quotient with the given potential $w(\Bx)=\|\Bx\|_2^2 + \delta \sum_{i=1}^d x_i^4$ and $\delta=0.1$ across different dimensions. The vertical line segments represent one standard deviation. }
		\label{fig:per_eigen}
\end{figure}

\subsubsection{Development of postprocessing techniques with FEX} Here, we demonstrate how FEX can facilitate the development of postprocessing techniques to achieve more refined results. For instance, consider the case where of $w=\|\Bx\|_2^2$ and $\Omega=\mathbb{R}^d$. In this scenario, the smallest eigenvalue of (\ref{eqn:eigenvalue}) is $d$, and the corresponding eigenfunction is $\exp(-\frac{\|\Bx\|_2^2}{2})$. We examine dimensions $d=2,4,6,8,10$. Given the lower dimensionality compared to the previous section, we can use smaller batch sizes: specifically, the batch size for estimating the first and third terms of the objective function~\eqref{eqn:eigenvalueloss} is 10,000, while that for the second term (boundary) is 2,000. 

 FEX discovers a high probability to have the ``$\exp$'' operator at the tree root ($100$\% for $d=2$, $4$, $6$, $8$, and $93.3$\% for $d=10$ as shown in Figure~\ref{fig:eigenvalues}).  Therefore, it is reasonable to assume that the eigenfunction is of the form $\exp(v(\Bx))$. Let $u(\Bx)$ be $\exp(v(\Bx))$ and then Eqn.~\eqref{eqn:eigenvalue} is simplified to 
\begin{align}
    -\Delta v-\|\nabla v\|_2^2 + \|\Bx\|_2^2 = \gamma.
    \label{eqn:eigensimplify}
\end{align} 
Eqn.~\ref{eqn:eigensimplify} does not have a trivial zero solution so we can avoid the integration constraint used in Eqns.~ \eqref{eqn:shrodinger} and \eqref{eqn:eigenvalueloss}, which leads to Monte-Carlo errors.  Using Eqn. \eqref{eqn:eigensimplify} and the Rayleigh quotient $\mathcal{I}$, the values of $v$ and $\gamma$ are alternatively updated until they reach convergence.  The detail of this iterative algorithm is presented in Appendix~\ref{sec:iterative}. Figure~\ref{fig:eigenvalues} shows the relative absolute error of the estimated eigenvalues with the growth of the dimensions. We can see that directly optimizing~\eqref{eqn:eigenvalueloss} with the NN method produces a large error on the eigenvalue estimation, especially when the dimension is high (e.g., the relative error is up to 20\% when $d=10$). With the postprocessing algorithm with FEX, we can identify the eigenvalue with an error close to zero.

\begin{figure}[ht]
		\centering
		\includegraphics[width=0.9\linewidth]{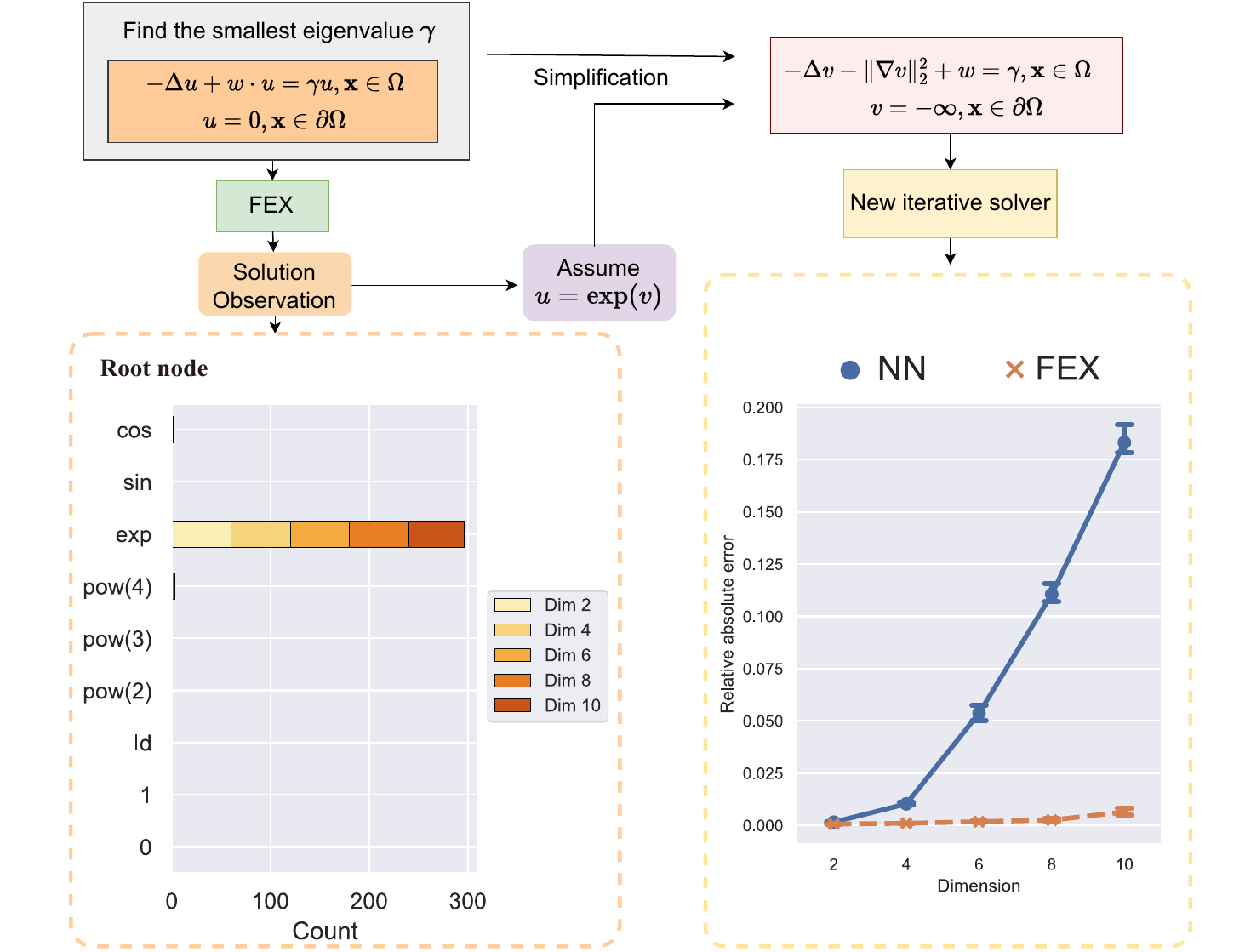}
		\caption{Eigenvalue problems and postprocessing algorithm design with FEX. Bottom left: We observed that the exponent operator ``$\exp(\cdot)$'' dominates the tree root in the FEX searching loop. Based on this observation, we assume the solution is $\exp(v(\Bx))$ and simplify the original PDE to a new PDE that avoids the trivial solution. Bottom right: The NN-based method produces a large error on the eigenvalue estimation, especially when the dimension is high ($d = 10$).
With the postprocessing algorithm with FEX, we can identify the eigenvalue with an error close to zero.}
		\label{fig:eigenvalues}
\end{figure}

\section{Discussion and Perspectives}\label{sec:cond}

In this paper, we proposed the finite expression method - a methodology to find PDE solutions as simple mathematical expressions. Our theory showed that mathematical expressions can overcome the curse of dimensionality. We provided one implementation of representing mathematical expressions using trees and solving the formulated combinatorial optimization in FEM using reinforcement learning. Our results demonstrated effectiveness of FEM at achieving high, even machine-level accuracy on various high-dimensional PDEs while existing solvers suffered from low accuracy in comparison.

While the proposed FEX solver is capable of identifying finite expressions with machine-level accuracy for several classes of high-dimensional PDEs, its computational cost remains significant. This cost is primarily divided into two phases: the searching phase and the fine-tuning phase.

\begin{itemize}
\item While the fine-tuning phase in FEX resembles the procedure used in NN methods for optimizing the parameters of surrogate models, it is significantly more efficient. This efficiency arises from the substantially smaller number of parameters in the FEX method compared to NN models. Specifically, the number of parameters in FEX is approximately ${d\times2^{L/2-1}+2^{L/2+1}}$ where $d$ is the input dimension and $L$ is the depth of the tree.  In most of the experiments presented in Section~\ref{sec:numericalexp}, we used $L\leq 4$ . In contrast, the number of parameters for the NN model is about $dm+(\ell-2) m^2+m$
where $m$ represents the size of the hidden layer and
$\ell$ is the number of hidden layers. For the NN models used in Section~\ref{sec:numericalexp}, we set $m=50$ and $\ell=7$. 

\item The searching phase in the FEX method introduces a significant computational overhead. This phase involves RL-based optimization of the controller and the evaluation of operator sequence scores, both of which are computationally intensive. However, the method's flexibility in configuring the searching process allows for balancing cost and performance.

Users can adjust various hyperparameters to tailor the computational effort to specific problem requirements or resource constraints. These hyperparameters include the number of iterations for the RL-based optimization of the controller, the number of iterations used to estimate the score function for an operator sequence, the depth of the binary tree structure, and the batch size.

Furthermore, certain components of the FEX algorithm can be parallelized to enhance efficiency. For instance, the score estimation for each operator sequence within a batch can be performed in parallel, and the final candidate fine-tuning process is also amenable to parallelization.

    \item If a complete binary tree is used (consisting of alternating layers of unary and binary nodes), the number of operators grows approximately as $2^{L/2}$, where $L$ is the tree depth. Empirically, we observed an exponential increase in optimization time with increasing tree depth, which aligns with the exponential growth in the number of operators. However, in practice, it is often unnecessary to use extremely deep or complete binary trees. Instead, shallower or pruned trees can provide a good balance. 
\end{itemize}

While CO is inherently a complex challenge, continued exploration can be highly advantageous in achieving improved performance, particularly when dealing with more intricate PDE problems within the framework of FEX.

\begin{itemize}
    \item \textbf{Efficient computation of operator sequence score. } We employed Formulation~\eqref{eqn:score} as a proxy for assessing the quality of an operator sequence. However, it's important to note that this formulation necessitates optimization over multiple steps, resulting in non-negligible computational costs that can impact the overall expense of solving the CO problem in FEX. Therefore, it is crucial to define a more efficient scoring method that reduces computational expenses and simplifies the identification of favorable operator sequences.
    
    \item \textbf{Design of the controller.} As an illustrative example of CO solving in this work, we employed a straightforward fully connected network to model the distribution responsible for proposing operator selections. It's worth noting that more sophisticated techniques can also be applied, such as recurrent neural networks~\citep{petersen2021deep}, which take prior decisions as input and generate new decisions as output. These advanced methods can offer enhanced modeling capabilities and adaptability to the context of the problem.
\end{itemize}
\section*{Data and Code Availability} No date is generated in this work. Source codes for reproducing the results in this paper are available online at: \url{https://github.com/LeungSamWai/Finite-expression-method}. The source codes are released under MIT license.


\acks{H. Yang was partially supported by the US National Science Foundation under awards DMS-2244988, DMS-2206333, the Office of Naval Research Award N00014-23-1-2007, and the DARPA D24AP00325-00. S. Liang acknowledges partial support from a startup grant by Texas Tech University. S. Liang acknowledges the helpful suggestion from Dr. Yuanran Zhu from Lawrence Berkeley National Laboratory on the eigenvalue problems.}


\newpage

\appendix
\section{Proofs of Theorems} \label{sec:proof}
\begin{proof}[Proof of Theorem \ref{thm:main}]
By Theorem 1.1 of \cite{Shen2021DeepNA}, for any $f\in C([a,b]^d)$ as a continuous function on $[a,b]^d$, there exists a fully connected neural network (FNN) $\phi$ with width $N=36d(2d+1)$ and depth $L=11$ (i.e., $11$ hidden layers) such that, for an arbitrary  $\varepsilon>0$, $\|\phi-f\|_{L^\infty([a,b]^d)}<\varepsilon$. This FNN is constructed via an activation function with an explicit formula $\sigma(x)=\sigma_1(x):=\big|x-2\lfloor \tfrac{x+1}{2}\rfloor\big|$ for $x\in[0,\infty)$ and $\sigma(x)=\sigma_2(x):=\frac{x}{|x|+1}$ for $x\in(-\infty,0)$. Therefore, $\sigma(x) = \frac{\text{sign}(x)+1}{2}\sigma_1(x) - \frac{\text{sign}(x)-1}{2}\sigma_2(x)$. Hence, it requires at most $18$ operators to evaluate $\sigma(x)$. For an FNN of width $N$ and depth $L$, there are $N(d+1)+(L-1)N^2$ operators ``$\times$", $Nd-1+(L-1)N(N-1)$ operators ``$+$", and $NL$ evaluations of $\sigma(x)$ to evaluate an output of the FNN. Therefore, the FNN $\phi$ is a mathematical expression with at most $k_d:=103680d^4 + 103824d^3 + 39600d^2 + 6804d - 1=\mathcal{O}(d^4)$ operators. Therefore, for any $\varepsilon>0$, any continuous function $f$ on $[a,b]^d$, there is a $k_d$-finite expression that can approximate $f$ uniformly well on $[a,b]^d$ within $\varepsilon$ accuracy. Since $k_d$ is independent of $\varepsilon$, it is clear that the function space of $k_d$-finite expressions is dense in $C([a,b]^d)$.
\end{proof}

\begin{proof}[Proof of Theorem \ref{thm:main2}]
By Cor. 3.8 of~\cite{Jiao2021DeepNN}, let $p\in[1,+\infty)$, for any $f\in \mathcal{H}^\alpha_\mu([0,1]^d)$ and $\varepsilon>0$, there exists an FNN $\phi$ with width
\begin{align*}
N=\max\{2d\left\lceil \log_2\left(\sqrt{d}\left( \frac{3\mu}{\varepsilon}\right)^{1/\alpha}\right)\right\rceil,2\left\lceil \log_2\frac{3\mu d^{\alpha/2}}{2\varepsilon} \right\rceil+2\}
\end{align*}
and depth $L=6$ such that $\|\phi-f\|_{L^p([0,1]^d)}<\varepsilon$. This FNN is constructed via activation functions chosen from the set $\{\sin(x),\max\{0,x\},2^x\}$. Similar to the proof for Theorem \ref{thm:main}, there are $N(d+1)+(L-1)N^2$ operators ``$\times$", $Nd-1+(L-1)N(N-1)$ operators ``$+$", and $NL$ evaluations of activation functions to evaluate an output of the FNN. Therefore, the total number of operators in $\phi$ as a mathematical expression is $\mathcal{O}(d^2(\log d+\log\frac{1}{\varepsilon})^2)$, which completes the proof.
\end{proof}

\section{Algorithm to Solve CO with Reinforcement Learning}\label{appendix:code}
We have included the pseudo-code for the proposed FEX implementation in Algorithm~\ref{alg:workflow} and Algorithm~\ref{alg:expandingtree}. Specifically, Algorithm~\ref{alg:workflow} outlines the procedure for solving the CO problem using a predefined fixed tree. On the other hand, Algorithm~\ref{alg:expandingtree} is designed to iteratively expand a tree, thereby enhancing its expressiveness in the pursuit of identifying an improved solution based on the approach described in Algorithm~\ref{alg:workflow}. 

\begin{algorithm}[ht]  
    \caption{FEX with a fixed tree}  
    \label{alg:workflow} 
    \textbf{Input:} PDE and the associated functional $\mathcal{L}$; A tree $\mathcal{T}$; Searching loop iteration $T$; Coarse-tune iteration $T_1$ with Adam; Coarse-tune iteration $T_2$ with BFGS; Fine-tune iteration $T_3$ with Adam; Pool size $K$; Batch size $N$.
    
    \textbf{Output:} The solution $u(\bm{x};\mathcal{T},\hat{\Be},\hat{\bm{\theta}})$.
    \begin{algorithmic}[1]
        \State Initialize the agent $\bm{\chi}$ for the tree $\mathcal{T}$
        \State $\mathbb{P} \gets \{\}$
        \For{$\_$ from $1$ to $T$}
        \State Sample $N$ sequences $\{\Be^{(1)}, \Be^{(2)}, \cdots, \Be^{(N)}\}$ from $\bm{\chi}$
        \For{$n$ from $1$ to $N$}
        \State Optimize $\mathcal{L}(u(\bm{x};\mathcal{T},\Be^{(n)},\bm{\theta}))$ through coarse-tuning over $T_1+T_2$ iterations to get score $S(\Be^{(n)})$
        \If{$\Be^{(n)}$ belongs to the top-$K$ of all scorings in $\mathbb{P}$}
        \State $\mathbb{P}$.append($\Be^{(n)}$)
        \State $\mathbb{P}$ pops some $\Be$ of the smallest score when $|\mathbb{P}|>K$
        \EndIf
        \EndFor
        \State Update $\bm{\chi}$ using~\eqref{eqn:mcmcrisk}
        \EndFor
            \For{$\Be$ in $\mathbb{P}$}  
                \State Fine-tune $\mathcal{L} (u(\bm{x};\mathcal{T},\Be,\bm{\theta}))$ with $T_3$ iterations.
            \EndFor
        \State \Return the expression with the smallest fine-tune error. 
    \end{algorithmic}
\end{algorithm} 

\begin{algorithm}[ht]  
    \caption{FEX with progressively expanding trees}  
    \label{alg:expandingtree} 
    \textbf{Input:} Tree set $\{\mathcal{T}_1, \mathcal{T}_2, \cdots \}$; Error tolerance $\epsilon$;
    
    \textbf{Output:} the solution $u(\bm{x};\hat{\mathcal{T}},\tilde{\Be},\tilde{\bm{\theta}})$.
    \begin{algorithmic}[1]
        \For{$\mathcal{T}$ in $\{\mathcal{T}_1, \mathcal{T}_2, \cdots \}$}  
        \State Initialize the agent $\bm{\chi}$ for the tree $\mathcal{T}$
        \State Obtain $u(\bm{x};\mathcal{T},\hat{\Be},\hat{\bm{\theta}})$ from Algorithm~\ref{alg:workflow}
        \If{$\mathcal{L}(u(\cdot;\mathcal{T},\hat{\Be},\hat{\bm{\theta}}))$ $\leq \epsilon$}
        \State Break
        \EndIf
        \EndFor
        \State \Return the expression with the smallest functional value. 
    \end{algorithmic}  
\end{algorithm} 

\section{Using a Fully Connected Network to Model a Controller}
\label{appendix:controller}
We illustrate the use of a fully connected neural network to model the controller, which generates probability mass functions for each node within a tree, in Figure~\ref{fig:fc_controller}. As an example, assume that we have a tree with three nodes: $n_1=2$ nodes for the binary set with $|\mathbbm{B}|=2$ operators, and $n_2=1$ node for the unary set with $|\mathbbm{U}|=3$ operators. Consequently, the output size of the neural network is $n_1|\mathbbm{B}|+n_2|\mathbbm{U}|=7$. This output is divided into $n_1+n_2$ parts, each corresponding to the probability mass function for its respective node. In addition, the input to the neural network is simply a constant vector.

\begin{figure}[ht]
		\centering

        \includegraphics[width=0.5\linewidth]{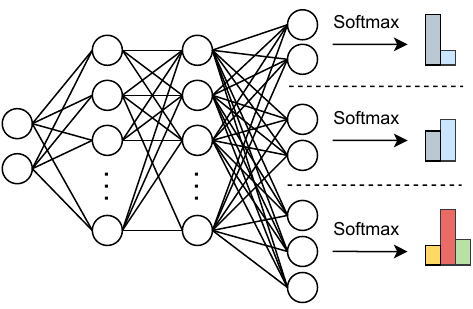}
		\caption{Illustration of using a fully connected neural network to model the controller that outputs the probability mass functions for each of the node within a tree. As an example, the tree contains three nodes: $n_1=2$ nodes for the binary set with $|\mathbbm{B}|=2$ operators and $n_2=1$ nodes for the unary set with $|\mathbbm{U}|=3$ operators. }
		\label{fig:fc_controller}
\end{figure}

\section{Eigenvalue Reference}\label{sec:eignref} 
 
When $w(\Bx)=\|\Bx\|_2^2$, the smallest eigenvalue of (\ref{eqn:eigenvalue}) is $d$ and the associated eigenfunction is $\phi(\Bx):=\exp(-\frac{\|\Bx\|_2^2}{2})$. When considering $w(\Bx)=\|\Bx\|_2^2+\delta \sum_{i=1}^d x_i^4$, the Rayleigh quotient of $\phi$ gives 
\begin{align}
\label{eqn:eigenvalue771}
\begin{split}
\mathcal{I}(\phi)&=\frac{\int_\Omega \|\nabla \phi\|_2^2dx+\int_\Omega (\|\Bx\|_2^2+\delta \sum_{i=1}^d x_i^4)  \phi^2 dx}{\int_\Omega \phi^2 dx}
\\&=d+\delta\frac{\int_\Omega \sum_{i=1}^d x_i^4  \phi^2 dx}{\int_\Omega \phi^2 dx}
\\&=d+\delta \frac{\int_\Omega \sum_{i=1}^d x_i^4  \exp(-\|\Bx\|_2^2) dx}{\int_\Omega \exp(-\|\Bx\|_2^2) dx}
\\&=d+\delta d \frac{\int_{-\infty}^\infty x_i^4\exp(-x_i^2)dx }{\int_{-\infty}^\infty \exp(-x_i^2)dx}=d+\delta d \frac{\Gamma(5/2)}{\Gamma(1/2)}=d+3/4 \delta d,
\end{split}
\end{align}
where $\Gamma$ represents the Gamma function. Therefore, the smallest eigenvalue should be less than $\mathcal{I}(\phi)=d+3/4 \delta d$.

\section{FEX-inspired Iterative Method for Eigenpairs}
\label{sec:iterative}
First, by solving Eqn.~\eqref{eqn:eigenvalueloss} with our FEX, we obtain an estimated eigenfunction $u(\Bx; \mathcal{T}, \Bhe, \hat{\bm{\theta}})$ and get the estimation of the eigenvalue through the Rayleigh quotient $\gamma_0 = \mathcal{I}(u(\cdot; \mathcal{T}, \Bhe, \hat{\bm{\theta}}))$. Then we can utilize Eqn.~\eqref{eqn:eigensimplify} to iteratively find the eigenpair. We define the function for Eqn.~\eqref{eqn:eigensimplify} by 
\begin{align}
    \mathcal{L}_2(v, \gamma):=\big\|-\Delta v-\|\nabla v\|_2^2 + \|\Bx\|_2^2- \gamma\big\|_{L_2(\Omega)}^2.
    \label{eqn:simpliedfunctional}
\end{align}
Given $\gamma_i$, we aim to find $v$ that is expressed by mathematical expression and minimizes $\mathcal{L}_2(v, \gamma_i)$. Assume $v$ is expressed by a binary tree ($v:=v(\cdot; \mathcal{T}, \Be, \bm{\theta})$), and then we can search the solution using our FEX with the following optimization, 
\begin{align}
    \Be_i^*, \bm{\theta}_i^* \approx \argmin_{\Be, \bm{\theta}}\mathcal{L}_2(v(\cdot; \mathcal{T}, \Be, \bm{\theta}), \gamma_i).
    \label{eqn:simpliedfunctionaliterative}
\end{align}
Next, we can calculate the current estimated eigenvalue by $\gamma_{i+1} =\mathcal{I}(\exp(v(\cdot; \mathcal{T}, \Be_i^*, \bm{\theta}_i^*)))$.

If continuing this loop for $G$ times, we will obtain the eigenpair $\gamma_{G}$ and $\exp(v(\cdot; \mathcal{T}, \Be_G^*, \bm{\theta}_G^*))$. 

For implementation, the number of the iterative loop is $G=10$. $\lambda_1=\lambda_2=500$ in ~\eqref{eqn:eigenvalueloss}. The batch size for estimating the first term and third term of~\eqref{eqn:eigenvalueloss} is 10,000 while that of the second term (boundary) is 2,000. ReLU$^2$ is used in ResNet, following~\cite{yu2018deep}.

\vskip 0.2in
\bibliography{reference}

\end{document}